\documentclass[11pt]{article}
\usepackage{amsmath,amsthm,amsfonts,amssymb,amscd, amsxtra}

\newcommand{\banacha}{\mathbb X}
\newcommand{\banachb}{\mathbb Y}

\newtheorem{theorem}{Theorem}
\newtheorem{lemma}[theorem]{Lemma}

\newtheorem{corollary}[theorem]{Corollary}
\newtheorem{proposition}[theorem]{Proposition}

\newtheorem{remark}{Remark}
\begin{document}
\title{Local convergence  of  inexact Gauss-Newton like methods for injective-overdetermined systems of equations \\ under a majorant condition}

\author{ M.L.N. Gon\c calves \thanks{IME/UFG, Campus II- Caixa
    Postal 131, CEP 74001-970 - Goi\^ania, GO, Brazil (E-mail:{\tt
      maxlng@ufg.br}).  The author was supported in part by CAPES, FAPEG/GO and CNPq Grant 471815/2012-8.} }
 \maketitle
\begin{abstract}
In this paper, inexact Gauss-Newton like methods for solving
injective-overdetermined systems of equations are studied.
We use a  majorant condition, defined by a function whose derivative is not necessarily convex, to
extend and improve  several existing results on the local convergence of the Gauss-Newton methods.
In particular, this analysis guarantees the convergence of the methods for two  important new cases.
\end{abstract}


\maketitle
\section{Introduction}\label{sec:int}
Let  $F$ be a continuously differentiable nonlinear function from a
open set  $\Omega$ in a real or complex Hilbert space $\banacha$ to
another  Hilbert space $\banachb$.  Consider the {\it systems of
nonlinear equations},
\begin{equation}\label{eq:p1}
F(x)=0.
\end{equation}

The Gauss-Newton method (see \cite{MAX2,MAX3,MAX5,MAX6}) is a generalized Newton method  for solving
such systems. It finds least-squares solutions of \eqref{eq:p1},
which may or may not be solutions of the original problem
\eqref{eq:p1}. These least-squares solutions  are stationary points
of the objective function of  the nonlinear least squares problem
$$
\min_{x\in \Omega } \;\|F(x)\|^2.
$$
This paper
is focused on the case in which the least-squares solutions of
\eqref{eq:p1}   also solve \eqref{eq:p1}. In the theory of nonlinear
least squares problems, this case is called the zero-residual case.

As it is well known, each Gauss-Newton iteration requires the
solution of a linear system involving $F'$. But this may be
computationally expensive, because in many cases even the
calculation of the derivative of F is hard to perform.
Thus, inexact variations of the method must be
considered for ensuring a good implementation. For more details on the different inexact versions of the Gauss-Newton and Newton methods, we refer the reader to \cite{AAA,MAX1,MAX3,FS1010,JM10,B10,1015} and the references therein.


Denote by $A^*$ the adjoint of the operator $A$. Formally, the
inexact Gauss-Newton like methods, which we will consider, are
described as follows: Given an initial point $x_0 \in {\Omega}$,
define
$$
x_{k+1}={x_k}+S_k,\qquad B(x_k)S_k=-F'(x_k)^*F(x_{k})+r_{k}, \qquad
k=0,1,\ldots,
$$
where  $B(x_k)$ is a suitable invertible approximation of the
derivative $F'(x_k)^*F'(x_{k})$, the residual tolerance,  $r_k$,
and the preconditioning invertible matrix, $P_{k}$, (considered for the first time in \cite{B10}) for the linear systems defining the step $S_k$  satisfy
$$
\|P_{k}r_{k}\|\leq \theta_{k}\|P_{k}F'(x_k)^*F(x_{k})\|,
$$
for a suitable forcing number $\theta_{k}$.
In particular, the above
process is  {the inexact modified Gauss-Newton method} if
$B(x_k)=F'(x_0)^*F'(x_0)$. It is {the Gauss-Newton like
method } if $\theta_{k}=0$, it corresponds to the { inexact Gauss-Newton
method} if $B_k=F'(x_k)^TF'(x_k)$, and it represents the {
Gauss-Newton method} if $\theta_{k}=0$ and $B_k=F'(x_k)^TF'(x_k)$.
It is also worth to  point out that if $F'(x)$ is invertible for all $x \in \Omega$, the inexact Gauss-Newton like methods become
{the inexact Newton-like methods}, which in particular include the inexact modified Newton, Newton-like, inexact Newton and Newton methods.

Results of convergence for the Gauss-Newton methods  have been discussed  by many authors, see for
example \cite{Argyros2010,Argyros,AAA,C11,F08,F10,MAX1,MAX2,MAX3,MAX4,FS1010,FS06,MAX5,MAX6,huangy2004,LZJ,JM10,B10,A3,1015,S86,XW10,XW9,10100}.
Recent research  attempts to alleviate  the assumption of
Lipschitz continuity on the operator~$F'$. For this, the main techniques used are  the majorant
condition, see for example  \cite{F08,F10, MAX1,MAX2,MAX3,MAX4,FS1010,FS06,MAX5,MAX6},
and the generalized Lipschitz condition according to X.Wang, see
for example \cite{AAA,C11,huangy2004, LZJ,1015,XW10,XW9,10100}.
But, if the derivative of the majorant function is not convex, these conditions are not equivalent and
the latter one can be seen as a particular case of the first one
(see the discussion on pp. 1522 of \cite{F10} or remark~4 in \cite{MAX5}). Moreover,  the majorant formulation provides a clear
relationship between the majorant function and the nonlinear
operator under consideration,  simplifying  the proof of convergence
substantially.

In what follows, under a milder majorant condition, we present a new
local convergence analysis of  inexact Gauss-Newton like methods
for solving \eqref{eq:p1}. The convergence for this family of
methods was obtained in \cite{MAX3} assuming that the derivative of
the majorant function is convex. In this paper an analogous result
is obtained without this hypothesis. We establish the well
definedness and the convergence, along with results on the
convergence rates.
The lack of convexity of the derivative of the majorant function here,
allows us to obtain two new important special cases, namely, the
convergence can be ensured under H\"{o}lder-like  and generalized
Lipschitz conditions.   In the latter case,  the results are
obtained without assuming that the function that defines the
condition is nondecreasing. Thus, Theorem~3.3 in
\cite{AAA} is generalized  for the zero-residual case. Moreover, it is worth to mention that
the  hypothesis of convex derivative of the majorant function or
nondecreasing of the function  which defines  the generalized
Lipschitz condition, are needed only to obtain the quadratic convergence
rate when the inexact Gauss-Newton like methods reduce to the
Gauss-Newton method.

The organization of the paper is as follows. Next, we list some notations and one basic result used in
our presentation.  In Section~\ref{lkant} we state the main result. In Section~\ref{sec:PR} some properties of a certain sequence associated to the majorant function
are established and the main relationships between the majorant
function and the nonlinear function $F$ are presented.  In Section
\ref{sec:proof} our main result is proven and some applications of
this result are obtained in Section~\ref{apl}. Some final remarks
are offered in Section~\ref{fr}.

\subsection{Notation and auxiliary results} \label{sec:int.1}
The following notations and results are used throughout our
presentation.   Let $\banacha$ and $\banachb$ be Hilbert spaces. The
open and closed balls at $a \in \banacha$ with radius $\delta>0$ are
denoted, respectively by
$$
B(a,\delta) :=\{ x\in \banacha ;\; \|x-a\|<\delta \}, \qquad
B[a,\delta] :=\{ x\in \banacha ;\; \|x-a\|\leqslant \delta \}.
$$
The set $\Omega\subseteq\banacha$ is an open set, the function
$F:\Omega\to \banachb$ is continuously differentiable, and $F'(x)$
has a closed image in $\Omega$. The condition number of a continuous
linear operator $A: \banacha \to \banachb$ is denoted by
$\mbox{cond}(A)$  and it is defined as
$$\mbox{cond}(A):=\|A^{-1}\|\|A\|.$$
Let $A: \banacha \to \banachb$ be a continuous and injective linear
operator with closed image. The Moore-Penrose inverse
$A^\dagger:\banachb \to \banacha$ of $A$ is defined by
$$
A^\dagger:=(A^*A)^{-1} A^*,
$$
where  $A^*$ denotes the adjoint of the linear operator $A$.

The next lemma  is proven in  \cite{G1} (see also,  \cite{W1}) for
an $m\times n$ matrix with $m\geq n$ and $rank(A)=rank(B)=n$. This
proof holds in a more general context as we will  state bellow.

\begin{lemma} \label{lem:ban2}
Let $A, B: \banacha \to \banachb$ be a continuous linear operator
with closed images. If $A$ is injective and $\|A^\dagger
\|\|A-B\|<1$, then $B$ is injective and $$ \|B^\dagger\|\leq
\frac{\|A^\dagger\|}{ 1- \|A^\dagger\|\|A-B\|}.
$$
\end{lemma}

\section{Local analysis for inexact Gauss-Newton like methods } \label{lkant}
In this section we state and prove the main result of the paper. It consists of a local theorem for the inexact Gauss-Newton like methods for solving \eqref{eq:p1}.  First,  some results of a certain sequence associated to the  majorant function are obtained.
Then,  we establish the main relationships between the majorant function and the nonlinear function $F$.
Finally, we show the well definedness and convergence  of  inexact Gauss-Newton like methods, along with some results on the convergence rates.
 The statement of the theorem~is:

\begin{theorem}\label{th:nt}
Let $\banacha$ and $\banachb$ be  Hilbert spaces,
 $\Omega\subseteq \banacha$ be an open set and
$F:{\Omega}\to \banachb$ be a continuously differentiable
function such that $F'$  has a closed image in $\Omega$. Let $x_* \in \Omega,$ $R>0$,
$
\beta:=\|F'(x_*)^{\dagger}\| $ and $ \kappa:=\sup \left\{ t\in [0, R): B(x_*, t)\subset\Omega \right\}.
$
Suppose that $F(x_*)=0$,
$F '(x_*)$ is injective
and there exists a
continuously differentiable function $f:[0,\; R)\to \mathbb{R}$ such that
  \begin{equation}\label{Hyp:MH}
\beta\left\|F'(x)-F'(x_*+\tau(x-x_*))\right\| \leq
f'\left(\|x-x_*\|\right)-f'\left(\tau\|x-x_*\|\right),
  \end{equation}
  for  all $\tau \in [0,1]$,  $x\in B(x_*, \kappa)$  and
\begin{itemize}
  \item[{\bf h1)}]  $f(0)=0$ and $f'(0)=-1$;
  \item[{\bf  h2)}]  $f'$ is   strictly increasing.
\end{itemize}
Take $0\leq \vartheta<1$, $0\leq \omega_{2}<\omega_{1}$ such that  $\omega_{1}\vartheta+\omega_{2}<1 $.
Let $\nu  :=\sup \left\{t \in[0, R):f'(t)<0\right\},$
$$ \rho :=\sup \left\{\delta \in(0, \nu):{(1+\vartheta)\omega_{1}[f(t)/f'(t)-t]}/{t}+\omega_{1}\vartheta+\omega_{2}<1,\; t \,\in \,(0, \delta)\right\},\qquad
 r :=\min  \left\{\kappa, \, \rho \right\}.  $$
Consider the inexact Gauss-Newton like methods, with initial point $x_0\in
B(x_*, r)\backslash \{x_*\}$, defined by
$$
x_{k+1}={x_k}+S_k, \qquad  B(x_k)S_k=-F'(x_k)^*F(x_k)+r_{k}, \qquad
\; k=0,1,\ldots,
$$
where $B(x_k)$ is an invertible approximation of $F'(x_k)^*F'(x_k)$ satisfying the following conditions
$$
\|B(x_k)^{-1}F'(x_k)^*F'(x_k)\| \leq \omega_{1}, \qquad
\|B(x_k)^{-1}F'(x_k)^*F'(x_k)-I\| \leq \omega_{2}, \qquad \;
k=0,1,\ldots,
$$
and the residual tolerance, $r_k$, the forcing term, $\theta_k$,
and the preconditioning invertible matrix, $P_{k}$,  are such that
$$
\|P_{k}r_{k}\|\leq \theta_{k}\|P_{k}F'(x_{k})^*F(x_{k})\|, \qquad
0\leq \theta_{k}\mbox{cond}(P_{k}F'(x_{k})^*F'(x_{k}))\leq
\vartheta,\qquad \; k=0,1,\ldots.$$
 Define the scalar sequence
$\{t_k\}$, with initial point  $t_0=\|x_0-x_*\|$, by
$$
   t_{k+1} =(1+\vartheta)\omega_{1}|t_k-f(t_k)/f'(t_k)|+(\omega_{1}\vartheta+\omega_{2})t_k, \qquad
   k=0,1,\ldots\,.
$$
Then, the sequences $\{x_k\}$  and $\{t_k\}$ are well defined; $\{t_k\}$  is strictly decreasing, contained in $(0, r)$ and it converges to $0$. Furthermore, $\{x_k\}$ is contained in $B(x_*,r)$, it converges to  $x_*$ and there~hold:
$$
    \limsup_{k \to \infty} \; [{\|x_{k+1}-x_*\|}\big{/}{\|x_k-x_*\|}]\leq\omega_{1}\vartheta+\omega_{2},\qquad \lim_{k \to \infty}[{t_{k+1}}\big{/}{t_k}]=\omega_{1}\vartheta+\omega_{2}.
$$
If, additionally, given $0\leq p\leq1$
\begin{itemize}
  \item[{\bf  h3)}]  the function  $(0,\, \nu) \ni t \mapsto [f(t)/f'(t)-t]/t^{p+1}$ is  strictly increasing,
\end{itemize}
then the sequences  $\{x_k\}$  and $\{t_k\}$ satisfy
\begin{equation}\label{cd12}
\|x_{k+1}-x_*\| \leq
(1+\vartheta)\omega_{1}\left(\frac{f(t_0)-t_0f'(t_0)}{{t_0}^{p+1}f'(t_0)}\right)\|x_k-x_*\|^{p+1}+\left(\omega_{1}\vartheta+\omega_{2}\right){\|x_k-x_*\|},
\end{equation}
for all $\; k=0,1,\ldots,$ and
\begin{equation}\label{cd1}
\|x_{k}-x_*\|\leq t_k, \qquad k=0,1, \ldots.
\end{equation}
\end{theorem}

The above theorem extends and improves previous results of the
Gauss-Newton methods for solving injective-overdetermined systems of equations, in particular those in~\cite{AAA,F10,MAX3,MAX5}.

\begin{remark}
In particular,  we obtain, from  Theorem~\ref{th:nt},  the
convergence for inexact modified Gauss-Newton method if
$B(x_k)=F'(x_0)^*F'(x_0)$, Gauss-Newton like method if
$\vartheta=0$, inexact Gauss-Newton method if $\omega_1=1$ and
$\omega_2=0$,   and Gauss-Newton method if $\vartheta=0$,
$\omega_1=1$ and $\omega_2=0$. In this latter case, it is possible to prove
that  $r$ is the optimal convergence radius, see Theorem~2 in~\cite{MAX5}.
\end{remark}

\begin{remark}
If $F'(x_*)$ is invertible,  we obtain, from Theorem~\ref{th:nt},
the local convergence of the inexact Newton-Like methods for solving
nonlinear equations, which in particular imply the
convergence for the inexact modified Newton,  Newton-like, inexact
Newton and Newton  methods. In the last case, Theorem~\ref{th:nt} is similar to  Theorem~2 in \cite{F10}.
\end{remark}

\begin{remark}
In particular, if $f'$ is convex,  we can prove that {\bf h3} holds
for $p=1$ and, therefore, in this case, we are led to the result
proven in Theorem~7  of \cite{MAX3} with $c=0$.   Thus, the additional assumption that the majorant
function, $f$,  has convex derivative, is only necessary  in order to
obtain the inequality in \eqref{cd12} and \eqref{cd1}, which for the  Gauss-Newton method imply quadratic convergence rate.
More details on the convergence of the Gauss-Newton and Newton methods under a weaker majorant condition can be found in \cite{F10,MAX5}.
\end{remark}


From now on, we assume that all the assumptions of Theorem \ref{th:nt}
hold, with the exception of {\bf h3},  which will be considered to hold only when explicitly stated.
\subsection{Preliminary results} \label{sec:PR}
In this section, we will prove all the statements in Theorem~\ref{th:nt} regarding  the sequence $\{t_k\}$ associated to the  majorant function. The main relationships between the  majorant function  and the nonlinear operator will be also established.
\subsubsection{The scalar sequence} \label{sec:PMF}
In this part,  we will check the statements in Theorem~\ref{th:nt}
involving  $\{t_k\}$. We begin by proving  that $ \kappa$ and $\nu$ are positive.

\begin{proposition}  \label{pr:incr1}
The constants $ \kappa$ and $\nu $ are positive and $t-f(t)/f'(t)<0,$ for all $t\in (0,\,\nu).$
\end{proposition}
\begin{proof}
Since $\Omega$ is open and $x_*\in \Omega$, we can immediately conclude that $\kappa>0$. As $f'$ is continuous in $0$ with  $f'(0)=-1$, there exists $\delta>0$ such that $f'(t)<0$ for all $t\in (0,\, \delta).$ Hence,  $\nu>0$.

It remains to show that $t-f(t)/f'(t)<0,$ for all $t\in (0,\,\nu).$ Since $f'$ is strictly increasing,  $f$ is strictly convex. So,
$
0=f(0)>f(t)-tf'(t),
$
for $t\in  (0,\, R).$
If $t\in (0, \,\nu)$ then $f'(t)<0$, which, combined with the last inequality, yields the desired inequality.
\end{proof}
According to {\bf h2} and the definition of $\nu$, we have  $f'(t)< 0$ for all
$t\in[0, \,\nu)$.  Therefore, the Newton iteration map for  $f$ is well defined in
$[0,\, \nu)$. Let us call it $n_f$:
\begin{equation} \label{eq:def.nf}
  \begin{array}{rcl}
  n_f:[0,\, \nu)&\to& (-\infty, \, 0]\\
    t&\mapsto& t-f(t)/f'(t).
  \end{array}
\end{equation}
\begin{proposition}  \label{pr:incr3}
$
\lim_{t\to 0}|n_f(t)|/t=0.
$
As a consequence,  $\rho>0 $ and
\begin{equation}\label{eq:001}
0<\omega_{1}(1+\vartheta)|n_f(t)|+(\omega_{1}\vartheta+\omega_{2})t<t,\qquad t\in (0, \, \rho).
\end{equation}
\end{proposition}
\begin{proof}
Using the definition \eqref{eq:def.nf},  Proposition \ref{pr:incr1},  $f(0)=0$, and the definition of $\nu$, some algebraic manipulations give
\begin{equation} \label{eq:rho}
\frac{|n_f(t)|}{t}= [f(t)/f'(t)-t]/t=\frac{1}{f'(t)} \frac{f(t)-f(0)}{t-0}-1, \qquad t\in (0,\,\nu).
\end{equation}
As  $f'(0)=-1\neq 0$, the first statement follows by taking limit in~\eqref{eq:rho}, as $t$ goes to $0$.

Now, since $[1-(\omega_{1}\vartheta+\omega_{2})]/[\omega_{1}(1+\vartheta)]>0$, using the first statement of the proposition and \eqref{eq:rho},
 we conclude that there exists a
$\delta>0$ such that
$$
0<[f(t)/(tf'(t))-1]<[1-(\omega_{1}\vartheta+\omega_{2})]/[\omega_{1}(1+\vartheta)],
 \qquad   t\in (0, \delta),
$$
or, equivalently,
$$
0<\omega_{1}(1+\vartheta)[f(t)/f'(t)-t]/{t}+\omega_{1}\vartheta+\omega_{2}<1,
\qquad  t\in (0, \delta).
$$
Therefore,   $\rho$  is positive.

To prove \eqref{eq:001}, it is enough to use the definition of $\rho$ and  \eqref{eq:rho}.
\end{proof}
Using \eqref{eq:def.nf}, it is easy to see that  the sequence $\{t_k \}$ is equivalently defined as
\begin{equation} \label{eq:tknk}
 t_0=\|x_0-x_*\|, \qquad t_{k+1}=\omega_{1}(1+\vartheta)|n_f(t_k)|+(\omega_{1}\vartheta+\omega_{2})t_k, \qquad k=0,1,\ldots\, .
\end{equation}
\begin{corollary} \label{cr:kanttk}
The sequence $\{t_k\}$ is well defined,  strictly decreasing and contained in $(0, \rho)$. Moreover,  $\{t_k\}$ converges to $0$ with linear rate, i.e.,
$
\lim_{k\to \infty}t_{k+1}/t_k=\omega_{1}\vartheta+\omega_{2}.
$
\end{corollary}
\begin{proof}
Since $0<t_0=\|x_0-x_*\|<r\leq\rho$, using  Proposition~\ref{pr:incr3} and \eqref{eq:tknk}  it is simple to conclude that $\{t_k\}$ is well defined, strictly decreasing and contained in $(0, \rho)$. So, the first statement of the corollary holds.

Due to  $\{t_k \}\subset (0, \rho)$ is  strictly decreasing,  it
converges. So,  $\lim_{k\to \infty}t_{k}=t_*$ with $0\leq t_*<\rho$,
which, together with \eqref{eq:tknk}, implies $0\leq
t_{*}=\omega_{1}(1+\vartheta)|n_f(t_*)|+(\omega_{1}\vartheta+\omega_{2})t_*$.
But, if $t_*\neq 0$,  Proposition~\ref{pr:incr3}  implies
$\omega_{1}(1+\vartheta)|n_f(t_*)|+(\omega_{1}\vartheta+\omega_{2})t_*<t_*$,
hence $t_*=0$. Therefore, $t_k \rightarrow 0$.

 Now, using  $\lim_{k\to \infty}t_{k}=0$, the definition of $\{t_k\}$ in  \eqref{eq:tknk} and the first statement in Proposition~\ref{pr:incr3},
 we obtain that  $\lim_{k\to \infty}t_{k+1}/t_k=\lim_{k\to \infty}\omega_{1}(1+\vartheta)|n_f(t_k)|/t_k+\omega_{1}\vartheta+\omega_{2}=\omega_{1}\vartheta+\omega_{2}$.
  Hence, the linear rate is proved.
\end{proof}
\begin{proposition}  \label{pr:incr33}
If {\bf  h3} holds,  the function  $(0,\, \nu) \ni t \mapsto
|n_f(t)|/t^{p+1}$ is strictly increasing.
\end{proposition}
\begin{proof}
As $t-f(t)/f'(t)<0$ for all $t\in (0,\,\nu)$ (see Proposition~\ref{pr:incr1}),
the  result is an immediate consequence of
{\bf h3} and the definition in~\eqref{eq:def.nf}.
\end{proof}
\subsubsection{Relationship of the majorant function with the nonlinear function} \label{sec:MFNLO}
In this part we present the main relationships between the majorant
function, $f$, and the nonlinear function, $F$.
\begin{lemma} \label{wdns}
If \,$\| x-x_*\|<\min\{\nu,\kappa\}$, then
$F'(x)^* F'(x) $ is invertible and
$$
\left\|F'(x)^{\dagger}\right\|\leq {\beta}/{|f'(\|
x-x_*\|)|}.
 $$
In particular, $F'(x)^* F'(x)$ is invertible in $B(x_*, r)$.
\end{lemma}
\begin{proof}
As $\| x-x_*\|<\min\{\nu,\kappa\}$, we have $f'(\| x-x_*\|)<0$. Hence,  using the definition of $\beta$, the inequality \eqref{Hyp:MH} and {\bf h1},  we have
\begin{equation}\label{eq1010}
\|F'(x_*)^{\dagger}\|\|F'(x)-F'(x_*)\|=\beta\|F'(x)-F'(x_*)\|\leq
f'(\| x-x_*\|)-f'(0)< 1.
\end{equation}
Since $F'(x_*)$  is injective, \eqref{eq1010} implies, in view
of Lemma \ref{lem:ban2}, that  $F'(x)$ is injective. So, $F'(x)^*
F'(x) $ is invertible and, by the definition of
 $r$, we obtain that $F'(x)^* F'(x)$ is invertible for all $x\in B(x_*, r)$.
Moreover, from Lemma \ref{lem:ban2} we also have
$$
\left\|F'(x)^{\dagger}\right\|\leq \frac{\beta}{1-\beta\|F'(x)-F'(x_*)\| }\leq \frac{\beta}{1-(f'(\| x-x_*\|)-f'(0))}=\frac{\beta}{|f'(\| x-x_*\|)|} ,
 $$
where $f'(0)=-1$ and $f'<0$ in $[0,\nu)$ are used for obtaining the
last equality.
\end{proof}
Now, it is convenient to study the linearization error of $F$ at
point in~$\Omega$. For this we define
\begin{equation}\label{eq:def.er}
  E_F(x,y):= F(y)-\left[ F(x)+F'(x)(y-x)\right],\qquad y,\, x\in \Omega.
\end{equation}
We will bound this error by the error in the linearization of the
majorant function $f$
\begin{equation}\label{eq:def.erf}
        e_f(t,u):= f(u)-\left[ f(t)+f'(t)(u-t)\right],\qquad t,\,u \in [0,R).
\end{equation}
\begin{lemma}  \label{pr:taylor}
If  $\|x-x_*\|< \kappa$, then  $\beta \|E_F(x, x_*)\|\leq
e_f(\|x-x_*\|, 0).$
\end{lemma}
\begin{proof}
 Since   $B(x_*, \kappa)$ is convex,  we obtain that $x_*+\tau(x-x_*)\in B(x_*, \kappa)$, for $0\leq \tau \leq 1$.
 Thus,  as $F$ is  continuously differentiable in $\Omega$, the definition of $E_F$ and some simple manipulations yield
$$
\beta\|E_F(x,x_*)\|\leq  \int_0 ^1 \beta \left\|
[F'(x)-F'(x_*+\tau(x-x_*))]\right\|\,\left\|x_*-x\right\| \;
d\tau.
$$
From  the last inequality  and  assumption \eqref{Hyp:MH}, we obtain
$$
\beta\|E_F(x,x_*)\| \leq \int_0 ^1
\left[f'\left(\left\|x-x_*\right\|\right)-f'\left(\tau\|x-x_*\|\right)\right]\|x-x_*\|\;d\tau.
$$
Evaluating the above integral and using the definition of $e_f$, the
statement follows.
\end{proof}
Define the Gauss-Newton step for the functions $F$ by the following equality:
\begin{equation} \label{eq:ns}
S_{F}(x):=-F'(x)^{\dagger}F(x).
\end{equation}
\begin{lemma}  \label{passonewton}
If  $\|x-x_*\|< \min\{\nu, \kappa\}$, then $\|S_{F}(x)\|\leq |n_f(\|x-x_*\|)|+\|x-x_{*}\|.$
\end{lemma}
\begin{proof}
Using \eqref{eq:ns}, $F(x_*)=0$ and some algebraic manipulations, it follows from \eqref{eq:def.er} that
\begin{align*}
\|S_{F}(x)\|&=\|F'(x)^{\dagger}\left(F(x_{*})-[F(x)+F'(x)(x_{*}-x)]\right)+(x_{*}-x)\|\\
&\leq \|F'(x)^{\dagger}\|\|E_F(x,x_*)\|+\|x-x_{*}\|.
\end{align*}
So, the last inequality, together with the Lemmas~\ref{wdns} and \ref{pr:taylor}, gives
$$\|S_{F}(x)\|\leq
{e_f(\|x-x_*\|, 0)}/{|f'(\|x-x_*\|)|}+\|x-x_{*}\|.$$
Since $f'<0$ in $[0,\nu)$ and $\|x-x_*\|< \nu$, we obtain from the last inequality,  \eqref{eq:def.erf}  and {\bf h1}, that
$$\|S_{F}(x)\|\leq
{f(\|x-x_*\|)}/{f'(\|x-x_*\|)},$$
which, combined  with   \eqref{eq:def.nf} and Proposition~\ref{pr:incr1},  implies the desired inequality.
 \end{proof}
\begin{lemma} \label{l:wdef}
Let $\banacha$, $\banachb$, $\Omega$, $F$,  $x_*,$ $R,$ $\beta$ and  $\kappa$  as defined in Theorem~\ref{th:nt}.
Suppose that $F(x_*)=0$, $F '(x_*)$ is injective and there exists a continuously differentiable function $f:[0,\; R)\to \mathbb{R}$
 satisfying  \eqref{Hyp:MH}, {\bf h1} and {\bf h2}.  Let $\vartheta$,  $\omega_{1}$, $\omega_{2}$,  $\nu$, $\rho$ and $r$ as defined in Theorem~\ref{th:nt}. Assume that $x\in B(x_*, r)\backslash \{x_*\}$, i.e., $0<\|x-x_*\|< r$.   Define
\begin{equation} \label{eq:DNSqnG}
x_{+}={x}+S, \qquad  B(x)S=-F'(x)^*F(x)+r,
\end{equation}
where $B(x)$ is an invertible approximation of $F'(x)^*F'(x)$ satisfying the following conditions
\begin{equation}\label{con:qnG}
\|B(x)^{-1}F'(x)^*F'(x)\| \leq \omega_{1}, \qquad
\|B(x)^{-1}F'(x)^*F'(x)-I\| \leq \omega_{2},
\end{equation}
and the residual, $r$, the forcing term, $\theta$,  and the
preconditioning invertible matrix, $P$,  are such that
\begin{equation}\label{eq:ERROqnG}
\theta \mbox{cond}(P F'(x)^*F'(x))\leq\vartheta, \qquad  \|P r\|\leq \theta\|P F'(x)^*F(x)\|.
\end{equation}
Then $x_{+}$ is well defined and it holds
  \begin{equation}\label{tt1}
    \|x_{+}-x_*\| \leq \omega_{1}(1+\vartheta)|n_f(\|x-x_*\|)|+(\omega_{1}\vartheta+\omega_{2})\|x-x_*\|.
\end{equation}
As a consequence,
$$
\|x_{+}-x_{*}\|< \|x-x_*\|.
$$
\end{lemma}
\begin{proof}
First note that, as $\|x-x_*\|<r$, it follows from Lemma \ref{wdns} that $F'(x)^*F'(x)$ is invertible. Now, let $B(x)$ be an invertible approximation of it satisfying \eqref{con:qnG}. Thus, $x_{+}$ is well defined. Now, as $F(x_*)=0,$ some simple algebraic manipulations and \eqref{eq:DNSqnG} yield
$$
x_{+}-x_{*}=  x -x_* - B(x)^{-1}F'(x)^*\big(F(x)-F(x_*)\big)+B(x)^{-1}{r}.
$$
Using $F'(x)^*F'(x) F'(x)^\dagger=F'(x)^*$ and some algebraic manipulations, the above equation gives
\begin{multline*}
x_{+}-x_{*}=B(x)^{-1}F'(x)^*F'(x) F'(x)^\dagger\big(F(x_*)-[F(x)+F'(x)(x_{*}-x)]\big)+B(x)^{-1}{r}\\
+B(x)^{-1}\left(F'(x)^*F'(x)-B(x)\right)(x-x_*) .
\end{multline*}
The last equation, together with  \eqref{eq:def.er} and \eqref{con:qnG}, implies that
$$
\|x_{+}-x_{*}\|\leq \omega_1 \|F'(x)^\dagger\| \|E_{F}(x,x_{*})\|+\|B(x)^{-1}{r}\|
+\omega_2\|x-x_*\|.
$$
On the other hand, using \eqref{eq:ns}, \eqref{con:qnG} and \eqref{eq:ERROqnG} we obtain, by simple manipulations, that
\begin{align*}
\|B(x)^{-1}{r}\|&\leq\|B(x)^{-1}P^{-1}\|\|P{r}\|\\
&\leq\theta\|B(x)^{-1}F'(x)^*F'(x)\|\|(PF'(x)^*F'(x))^{-1}\|\|PF'(x)^*F'(x)\|\|F'(x)^\dagger F(x)\|\\
&\leq\omega_1\vartheta\|S_F(x)\|.
\end{align*}
Hence, it follows from the two last equations that
$$
\|x_{+}-x_{*}\|\leq \omega_1 \|F'(x)^\dagger\| \|E_{F}(x,x_{*})\|
+\omega_1\vartheta\|S_F(x)\|+\omega_2\|x-x_*\|.
$$
Combining the last equation with the Lemmas~\ref{wdns}, \ref{pr:taylor} and \ref{passonewton}, we obtain that
$$
\|x_{+}-x_{*}\|\leq \omega_1{e_f(\|x-x_*\|,
0)}/{|f'(\|x-x_*\|)|}+\omega_1\vartheta|n_f(\|x-x_*\|)|+(\omega_1\vartheta+\omega_2)\|x-x_*\|.
$$
Now, taking into account that $f(0)=0$, the definitions of $e_f$ and $n_f$  imply that
$${e_f(\|x-x_*\|,
0)}/{|f'(\|x-x_*\|)|}=|n_f(\|x-x_*\|)|.$$
So, the inequality in \eqref{tt1} follows by combining  the above two inequalities.

Take $x\in B(x_*, r)$. Since $0<\|x-x_*\|<r\leq \rho$, the inequalities in \eqref{eq:001} and  \eqref{tt1}    imply that
$$
\|x_+-x_*\|\leq  \omega_{1}(1+\vartheta)|n_f(\|x-x_*\|)|+(\omega_{1}\vartheta+\omega_{2})\|x-x_*\|<\|x-x_*\|,
 $$
 which proves the last statement of the lemma.
\end{proof}
\subsection{Inexact Gauss-Newton like sequence} \label{sec:proof}
In this section,  we will prove the statements in Theorem~\ref{th:nt} involving  the inexact Gauss-Newton like sequence $\{x_k\}$.
\begin{proposition}\label{pr:nthe}
The sequence $\{x_k\}$ is well defined, contained in $B(x_*,r)$ and
it converges to  $x_*$ with linear rate, i.e.,
\begin{equation} \label{eq:q2e}
    \limsup_{k \to \infty}\left[\|x_{k+1}-x_*\|\big{/}\|x_k-x_*\|\right]\leq\omega_1\vartheta+\omega_2.
  \end{equation}
If {\bf  h3} holds, the sequences  $\{x_k\}$  and $\{t_k\}$ satisfy
\begin{equation}\label{eq:mjs}
\|x_{k+1}-x_*\| \leq
(1+\vartheta)\omega_{1}\left(\frac{f(t_0)-t_0f'(t_0)}{{t_0}^{p+1}f'(t_0)}\right)\|x_k-x_*\|^{p+1}+\left(\omega_{1}\vartheta+\omega_{2}\right){\|x_k-x_*\|},
\end{equation}
for all $k=0,1,\ldots,$ and
\begin{equation}\label{eq:tk}
\|x_{k}-x_*\|\leq t_k, \qquad k=0,1, \ldots.
  \end{equation}
\end{proposition}
\begin{proof}
Since  $x_0\in B(x_*,r)/\{x_*\},$ i.e., $0<\|x_0-x_*\|<r,$ a  combination of
Lemma~\ref{wdns}, the last inequality in Lemma~\ref{l:wdef} and an induction argument, we conclude that  $\{x_k\}$ is well defined and it remains in $B(x_*,r)$.

We will now prove that  $\{x_k\}$ converges to  $x_*$.
Since $\|x_{k}-x_*\|<r\leq \rho$,  for $ k=0,1,\ldots \,$,  we obtain from
Lemma~\ref{l:wdef} with $x_+=x_{k+1},$ $x=x_{k},$ $r=r_{k},$ $B(x)=B(x_{k}),$ $P=P_k$ and $\theta=\theta_{k}$, and Proposition~\ref{pr:incr3}, that
\begin{equation}\label{eq:conv1}
0\leq\|x_{k+1}-x_*\|\leq \omega_{1}(1+\vartheta)|n_f(\|x_k-x_*\|)|+(\omega_{1}\vartheta+\omega_{2})\|x_k-x_*\|<\|x_{k}-x_*\|,\quad  k=0,1,\ldots \,.
\end{equation}

So, $\{\|x_{k}-x_*\| \}$ is a bounded and  strictly decreasing
sequence. Therefore $\{\|x_{k}-x_*\| \}$ converges. Let
$\ell_*=\lim_{k\to \infty}\|x_{k}-x_*\|$.
 Since  $\{\|x_{k}-x_*\| \}$ remains in $(0, \,\rho)$ and is strictly decreasing, we have $0\leq \ell_*<\rho$. Thus, taking
the limit in \eqref{eq:conv1} with $t$ converging to $0$ and using  the  continuity of $n_f$ in $[0, \rho)$,
we obtain that  $0\leq \ell_{*}=\omega_{1}(1+\vartheta)|n_f(\ell_*)|+(\omega_{1}\vartheta+\omega_{2})\ell_*$. But, if $\ell_* \neq 0$, Proposition~\ref{pr:incr3}  implies $\omega_{1}(1+\vartheta)|n_f(\ell_*)|+(\omega_{1}\vartheta+\omega_{2})\ell_*<\ell_*$,
hence $\ell_*=0$. Therefore,  the convergence $x_k \rightarrow x_*$ is proved.

In order to establish  inequality  \eqref{eq:q2e}, note that inequality \eqref{eq:conv1} implies
$$
\left[\|x_{k+1}-x_*\|\big{/}\|x_{k}-x_*\|\right]\leq \omega_{1}(1+\vartheta)\left[|n_f(\|x_{k}-x_*\|)|\big{/}\|x_{k}-x_*\|\right]+\omega_{1}\vartheta+\omega_{2}, \qquad k=0,1, \ldots.
$$
Hence, as  $\{\|x_{k+1}-x_*\|\big{/}\|x_{k}-x_*\|\}$
is bounded  and  $\lim_{k\to \infty}\|x_{k}-x_*\|=0$, the  desired inequality follows from the first statement in  Proposition~\ref{pr:incr3}.

Now,  inequality  \eqref{eq:conv1} also implies
\begin{equation}\label{c1}
\|x_{k+1}-x_*\|\leq \omega_{1}(1+\vartheta)\frac{|n_f(\|x_k-x_*\|)|}{\|x_k-x_*\|^{p+1}}\|x_k-x_*\|^{p+1}+(\omega_{1}\vartheta+\omega_{2})\|x_k-x_*\|,\quad  k=0,1,\ldots \,.
\end{equation}
Thus, using that $\{\|x_{k}-x_*\| \}$ is strictly decreasing,
Proposition~\ref{pr:incr33} and the definition of $n_f$ in
\eqref{eq:def.nf}, the last inequality gives  \eqref{eq:mjs}.

To end the proof, we will show  inequality \eqref{eq:tk} by induction. As $t_0=\|x_0-x_*\|$, it  is  immediate
for $k=0$. Assume that $\|x_{k}-x_*\|\leq t_k$.  Hence, \eqref{c1}, Proposition~\ref{pr:incr33} and the  definition of $t_{k+1}$ in \eqref{eq:tknk} imply
$$
\|x_{k+1}-x_*\|\leq \omega_{1}(1+\vartheta){|n_f(t_k)|}+(\omega_{1}\vartheta+\omega_{2})t_k=t_{k+1},
$$
Therefore,  inequality \eqref{eq:tk} holds.
\end{proof}
Proof of Theorem~\ref{th:nt} follows from Corollary~\ref{cr:kanttk} and Proposition~\ref{pr:nthe}.

\section{Special Cases} \label{apl}
In this section, we present some special cases of Theorem~\ref{th:nt}.
\subsection{Convergence results  under H\"{o}lder-like and Smale  conditions}
In this section, we  present  a local convergence theorem for the inexact Gauss-Newton like methods under a H\"{o}lder-like condition, see \cite{F10,MAX5,huangy2004}.  We also provide a Smale's theorem on the inexact Gauss-Newton like methods for analytical functions, cf. \cite{S86}.
\begin{theorem}\label{th:HV}
Let $\banacha$ and $\banachb$ be  Hilbert spaces,
 $\Omega\subseteq \banacha$ be an open set and
$F:{\Omega}\to \banachb$ be a continuously differentiable
function such that $F'$  has a closed image in $\Omega$. Let $x_* \in \Omega,$ $R>0$,
$
\beta:=\|F'(x_*)^{\dagger}\| $ and $ \kappa:=\sup \left\{ t\in [0, R): B(x_*, t)\subset\Omega \right\}.
$
Suppose that $F(x_*)=0$,
$F '(x_*)$ is injective
and there exists a constant $K>0$ and $ 0< p \leq 1$ such that
$$
\beta\left\|F'(x)-F'(x_*+\tau(x-x_*))\right\|\leq  K(1-\tau^p) \|x-x_*\|^p, \qquad   x\in B(x_*, \kappa) \quad \tau \in [0,1].
$$
Take $0\leq \vartheta<1$, $0\leq \omega_{2}<\omega_{1}$ such that  $\omega_{1}\vartheta+\omega_{2}<1 $.
Let $$r=\min \left\{\kappa, \left[\frac{(1-\omega_{1}\vartheta-\omega_{2})(p+1)}{K(1-\omega_{1}\vartheta-\omega_{2}+p(1+\omega_{1}-\omega_{2}))}\right]^{1/p}\right\}.$$
Consider the inexact Gauss-Newton like methods, with initial point $x_0\in
B(x_*, r)\backslash \{x_*\}$, defined by
$$
x_{k+1}={x_k}+S_k, \qquad  B(x_k)S_k=-F'(x_k)^*F(x_k)+r_{k}, \qquad
\; k=0,1,\ldots,
$$
where $B(x_k)$ is an invertible approximation of $F'(x_k)^*F'(x_k)$ satisfying the following conditions
$$
\|B(x_k)^{-1}F'(x_k)^*F'(x_k)\| \leq \omega_{1}, \qquad
\|B(x_k)^{-1}F'(x_k)^*F'(x_k)-I\| \leq \omega_{2}, \qquad \;
k=0,1,\ldots,
$$
and the residual tolerance, $r_k$, the forcing term, $\theta_k$, and
the preconditioning invertible matrix, $P_{k}$,  are such that
$$
\|P_{k}r_{k}\|\leq \theta_{k}\|P_{k}F'(x_{k})^*F(x_{k})\|, \qquad
0\leq \theta_{k}\mbox{cond}(P_{k}F'(x_{k})^*F'(x_{k}))\leq
\vartheta,\qquad \; k=0,1,\ldots.$$
 Define the scalar sequence
$\{t_k\}$, with initial point  $t_0=\|x_0-x_*\|$, by
$$
   t_{k+1} =\frac{(1+\vartheta)\omega_{1} \,p \,K t_{k}^{p+1}}{(p+1)[1-K\,t_k^{p}]}+(\omega_{1}\vartheta+\omega_{2})t_k, \qquad
   k=0,1,\ldots\,.
$$
Then, the sequences $\{x_k\}$  and $\{t_k\}$ are well defined; $\{t_k\}$  is strictly decreasing, contained in $(0, r)$ and it converges to $0$. Furthermore, $\{x_k\}$ is contained in $B(x_*,r)$, it converges to  $x_*$,
$$
\|x_{k+1}-x_*\| \leq
\frac{(1+\vartheta)\omega_{1} p\,K}{(p+1)[1- K\,\|x_0-x_*\|^{p}]}\|x_k-x_*\|^{p+1}+\left(\omega_{1}\vartheta+\omega_{2}\right){\|x_k-x_*\|},
\qquad k=0,1,\ldots,
$$
and
$$
\|x_{k}-x_*\|\leq t_k, \qquad k=0,1, \ldots.
$$
\end{theorem}
\begin{proof}
It is immediate to prove that  $F$, $x_*$ and $f:[0, \kappa)\to \mathbb{R}$, defined by
$
f(t)=Kt^{p+1}/(p+1)-t,
$
satisfy the inequality \eqref{Hyp:MH} and the conditions  {\bf h1}, {\bf h2} and  {\bf h3} in Theorem \ref{th:nt}.
In this case, it is easily seen that  $\rho$ and $\nu$, as defined in Theorem \ref{th:nt}, satisfy
$$
\rho=\left[\frac{(1-\omega_{1}\vartheta-\omega_{2})(p+1)}{K(1-\omega_{1}\vartheta-\omega_{2}+p(1+\omega_{1}-\omega_{2}))}\right]^{1/p} \leq \nu=[1/ K]^{1/p},
$$
and, as a consequence,  $r=\min \{\kappa,\; \rho\}$. Therefore, the statements of the theorem follow from
 Theorem~\ref{th:nt}.
\end{proof}
\begin{remark}
For $p=1$ in the previous theorem, we obtain the convergence of the
inexact Gauss-Newton like methods  under a Lipschitz condition, as obtained in Theorem~16 of \cite{MAX3} with $c=0$.
\end{remark}
Below, we present a theorem corresponding to Theorem \ref{th:nt} under  Smale's condition.
\begin{theorem}\label{theo:Smale}
Let $\banacha$ and $\banachb$ be  Hilbert spaces,
 $\Omega\subseteq \banacha$ be an open set and
$F:{\Omega}\to \banachb$  an analytic function such that $F'$  has a closed image in $\Omega$. Let $x_* \in \Omega,$ $R>0$,
$
\beta:=\|F'(x_*)^{\dagger}\| $ and $ \kappa:=\sup \left\{ t\in [0, R): B(x_*, t)\subset\Omega \right\}.
$
Suppose that $F(x_*)=0$,  $F '(x_*)$ is injective and
$$
 \gamma := \sup _{ n > 1 }\beta\left\| \frac
{F^{(n)}(x_*)}{n !}\right\|^{1/(n-1)}<+\infty.
$$
Take $0\leq \vartheta<1$, $0\leq \omega_{2}<\omega_{1}$ such that  $\omega_{1}\vartheta+\omega_{2}<1 $. Let
$a:=(1+\vartheta)\omega_{1}$, $b:=(1-\omega_{1}\vartheta-\omega_{2})$  and
$$
r:=\min \left\{\kappa, \frac{a+4b-\sqrt{(a+4b)^2-8b^2}}{4b\gamma}\right\}.
$$
Consider the inexact Gauss-Newton like methods, with initial point $x_0\in
B(x_*, r)\backslash \{x_*\}$, defined by
$$
x_{k+1}={x_k}+S_k, \qquad  B(x_k)S_k=-F'(x_k)^*F(x_k)+r_{k}, \qquad
\; k=0,1,\ldots,
$$
where $B(x_k)$ is an invertible approximation of $F'(x_k)^*F'(x_k)$ satisfying the following conditions
$$
\|B(x_k)^{-1}F'(x_k)^*F'(x_k)\| \leq \omega_{1}, \qquad
\|B(x_k)^{-1}F'(x_k)^*F'(x_k)-I\| \leq \omega_{2}, \qquad \;
k=0,1,\ldots,
$$
and the residual tolerance, $r_k$, the forcing term, $\theta_k$, and
the preconditioning invertible matrix, $P_{k}$,  are such that
$$
\|P_{k}r_{k}\|\leq \theta_{k}\|P_{k}F'(x_{k})^*F(x_{k})\|, \qquad
0\leq \theta_{k}\mbox{cond}(P_{k}F'(x_{k})^*F'(x_{k}))\leq
\vartheta,\qquad \; k=0,1,\ldots.$$
 Define the scalar sequence
$\{t_k\}$, with initial point  $t_0=\|x_0-x_*\|$, by
$$
   t_{k+1} =\frac{(1+\vartheta)\omega_{1} \gamma t_{k}^{2}}{2(1-\gamma\,t_k)^2-1}+(\omega_{1}\vartheta+\omega_{2})t_k, \qquad
   k=0,1,\ldots\,.
$$
Then, the sequences $\{x_k\}$  and $\{t_k\}$ are well defined; $\{t_k\}$  is strictly decreasing, contained in $(0, r)$ and it converges to $0$. Furthermore, $\{x_k\}$ is contained in $B(x_*,r)$, it converges to  $x_*$,
$$
\|x_{k+1}-x_*\| \leq
\frac{\gamma}{2(1-\gamma \|x_0-x_*\|)^2-1}\|x_k-x_*\|^2+\left(\omega_{1}\vartheta+\omega_{2}\right){\|x_k-x_*\|},
\qquad k=0,1,\ldots,
$$
and
$$
\|x_{k}-x_*\|\leq t_k, \qquad k=0,1, \ldots.
$$
\end{theorem}
\begin{proof}
In this case, the  real function, $f:[0,1/\gamma) \to \mathbb{R}$,
defined by $ f(t)={t}/{(1-\gamma t)}-2t$, is a majorant function for
the function $F$ on $B(x_*, 1/\gamma)$. Hence, as  $f$ has a convex
derivative, the proof follows the same pattern as outlined in
Theorem~18 of \cite{MAX3}.
\end{proof}

\subsection{Convergence result  under a generalized Lipschitz condition}
In this section, we  present a local convergence theorem for the
inexact Gauss-Newton like methods under a generalized Lipschitz
condition according to X.Wang (see \cite{huangy2004,XW10}). It is
worth to point out that the result in this section does not assume
that the function defining the generalized Lipschitz condition  is
nondecreasing. Thus,  Theorem~3.3 in \cite{AAA} is generalized for
the zero-residual case.
\begin{theorem} \label{th:XWT}
Let $\banacha$ and $\banachb$ be  Hilbert spaces,
 $\Omega\subseteq \banacha$ be an open set and
$F:{\Omega}\to \banachb$ be a continuously differentiable
function such that $F'$  has a closed image in $\Omega$. Let $x_* \in \Omega,$ $R>0$,
$
\beta:=\|F'(x_*)^{\dagger}\| $ and $ \kappa:=\sup \left\{ t\in [0, R): B(x_*, t)\subset\Omega \right\}.
$
Suppose that $F(x_*)=0$,
$F '(x_*)$ is injective
and there exists a  positive  integrable function $L:[0,\; R)\to \mathbb{R}$ such that
\begin{equation}\label{Hyp:XW}
\beta\left\|F'(x)-F'(x_*+\tau(x-x_*))\right\| \leq  \int^{\|x-x_*\|}_{\tau\|x-x_*\|} L(u){\rm d}u,
\end{equation}
for all $\tau \in [0,1]$, $x\in B(x_*, \kappa)$. Let
$
\bar{\nu}:=\sup \{t\in [0, R): \displaystyle \int_{0}^{t}L(u){\rm d}u-1 < 0\},
$
$$
\bar{\rho}:=\sup \left\{t\in (0, \delta):
\frac{(1+\vartheta)\omega_{1}\int^{t_k}_{0}L(u)u {\rm d}u}{t\left(1-\int^{t_k}_{0}L(u){\rm d}u\right)}+\omega_{1}\vartheta+\omega_{2}<1, \; t\in (0, \delta)\right\},
\qquad
\bar{r}=\min \left\{\kappa, \bar{\rho}\right\}.
$$
Consider the inexact Gauss-Newton like methods, with initial point $x_0\in
B(x_*, r)\backslash \{x_*\}$, defined by
$$
x_{k+1}={x_k}+S_k, \qquad  B(x_k)S_k=-F'(x_k)^*F(x_k)+r_{k}, \qquad
\; k=0,1,\ldots,
$$
where $B(x_k)$ is an invertible approximation of $F'(x_k)^*F'(x_k)$ satisfying the following conditions
$$
\|B(x_k)^{-1}F'(x_k)^*F'(x_k)\| \leq \omega_{1}, \qquad
\|B(x_k)^{-1}F'(x_k)^*F'(x_k)-I\| \leq \omega_{2}, \qquad \;
k=0,1,\ldots,
$$
and the residual tolerance, $r_k$, the forcing term, $\theta_k$, and
the preconditioning invertible matrix, $P_{k}$,  are such that
$$
\|P_{k}r_{k}\|\leq \theta_{k}\|P_{k}F'(x_{k})^*F(x_{k})\|, \qquad
0\leq \theta_{k}\mbox{cond}(P_{k}F'(x_{k})^*F'(x_{k}))\leq
\vartheta,\qquad \; k=0,1,\ldots.$$
 Define the scalar sequence
$\{t_k\}$, with initial point  $t_0=\|x_0-x_*\|$, by
$$
   t_{k+1} =\frac{(1+\vartheta)\omega_{1}\int^{t_k}_{0}L(u)u {\rm d}u}{1-\int^{t_k}_{0}L(u){\rm d}u}+(\omega_{1}\vartheta+\omega_{2})t_k, \qquad
   k=0,1,\ldots\,.
$$
Then, the sequences $\{x_k\}$  and $\{t_k\}$ are well defined; $\{t_k\}$  is strictly decreasing, contained in $(0, r)$ and it converges to $0$. Furthermore, $\{x_k\}$ is contained in $B(x_*,r)$, it converges to  $x_*$ and there~hold:
$$
    \limsup_{k \to \infty} \; [{\|x_{k+1}-x_*\|}\big{/}{\|x_k-x_*\|}]\leq\omega_{1}\vartheta+\omega_{2},\qquad \lim_{k \to \infty}[{t_{k+1}}\big{/}{t_k}]=\omega_{1}\vartheta+\omega_{2}.
$$
If, additionally, given $0\leq p\leq1$
\begin{itemize}
  \item[{ ${\bf h)}$}]  the function   $(0,\, \nu) \ni t \mapsto t^{1-p}L(t)$ is nondecreasing,
\end{itemize}
then the sequences  $\{x_k\}$  and $\{t_k\}$ satisfy
$$
\|x_{k+1}-x_*\| \leq
(1+\vartheta)\omega_{1}\left(\frac{\int^{t_0}_{0}L(u)u{\rm d}u}{{t_0}^{p+1}\left(1-\int^{t_0}_{0}L(u){\rm d}u\right)}\right)\|x_k-x_*\|^{p+1}+\left(\omega_{1}\vartheta+\omega_{2}\right){\|x_k-x_*\|},
$$
for all $\; k=0,1,\ldots,$ and
$$
\|x_{k}-x_*\|\leq t_k, \qquad k=0,1, \ldots.
$$
\end{theorem}
\begin{proof}
Let  ${\bar f}:[0, \kappa)\to \mathbb{R}$ be a differentiable function defined by
$$
{\bar f}(t)=\int_{0}^{t}L(u)(t-u){\rm d}u-t.
$$
Note  that the derivative of the function $f$ is given by
$$
 {\bar f}'(t)=\int_{0}^{t}L(u){\rm d}u-1.
$$
Since $L$ is integrable, ${\bar f}'$ is continuous (in fact ${\bar f}'$ is absolutely continuous). Hence, it is easy to see that \eqref{Hyp:XW} becomes
 \eqref{Hyp:MH} with $f'={\bar f}'$. Moreover, since $L$ is positive, the function $f={\bar f}$  satisfies
the conditions  {\bf h1} and  {\bf h2} in Theorem \ref{th:nt}. Direct algebraic manipulation yields
$$
\frac{1}{t^{p+1}} \left[\frac{{\bar f}(t)}{{\bar f}'(t)}-t\right]=\left[
\frac{1}{t^{p+1}}\displaystyle\int^{t}_{0}L(u)u{\rm d}u\right]
\frac{1}{|{\bar f}'(t)|}.
$$
If assumption {\bf  h} holds, then Lemma~$2.2$ of \cite{XW9} implies
that the first term on the right hand side of the above equation is
nondecreasing
  in $(0,\, \nu)$. Now,  since $1/|{\bar f}'|$ is  strictly  increasing in $(0,\, \nu)$, the above equation implies that {\bf  h3} in Theorem~\ref{th:nt},
  with $f={\bar f}$, also holds.
Therefore, the result  follows from Theorem~\ref{th:nt} with $f={\bar f}$, $\nu=\bar{\nu}$, $\rho=\bar{\rho}$ and $r=\bar{r}$.
\end{proof}
\begin{remark}
If $f'$ in Theorem~\ref{th:nt}  is convex,
the inequalities  \eqref{Hyp:MH} and \eqref{Hyp:XW} are equivalent. Otherwise,
the equivalence does not hold.
For more details see the discussion on pp. 1522 of \cite{F10} or remark~4 in~\cite{MAX5}.
\end{remark}

\section{Final remarks } \label{fr}
We presented, under a milder majorant condition,  a local convergence of  inexact Gauss-Newton like methods for solving
injective-overdetermined systems of equations.
Our main theorem was proved  without the  additional assumption that the majorant function  has convex derivative.
Among other things, the lack of this assumption allows us to obtain two new important special cases, namely, the convergence was ensured under H\"{o}lder-like  and  generalized Lipschitz conditions.

It would be interesting  to study results of semi-local convergence for the inexact Gauss-Newton like methods under a similar majorant condition. This analysis will be carried out in the future.



\def\cprime{$'$}

\end{document}